\documentclass[11 pt]{amsart}
\usepackage{amssymb}
\newtheorem{theorem}{Theorem}[section]
\newtheorem{lemma}[theorem]{Lemma}
\newtheorem{proposition}[theorem]{Proposition}
\newtheorem{corollary}[theorem]{Corollary}
\theoremstyle{definition}

\newtheorem{example}[theorem]{Example}

\theoremstyle{remark}
\newtheorem{remark}[theorem]{Remark}
\begin{document}

\title[Involutions on the product of two projective spaces]{Orbit spaces of free involutions on the product of two projective spaces}

\author{Mahender Singh}
\address{Institute of Mathematical Sciences\\ C I T Campus\\ Taramani\\ Chennai 600113\\ India}
\email{mahender@imsc.res.in}
\subjclass{Primary 57S17; Secondary 55R20, 55M20}
\keywords{Cohomology algebra, finitistic space, Index of involution, Leray spectral sequence}

\begin{abstract}
Let $X$ be a finitistic space having the mod 2 cohomology algebra of the product of two projective spaces. We study free involutions on $X$ and determine the possible mod 2 cohomology algebra of orbit space of any free involution, using the Leray spectral sequence associated to the Borel fibration $X \hookrightarrow X_{\mathbb{Z}_2} \longrightarrow B_{\mathbb{Z}_2}$. We also give an application of our result to show that if $X$ has the mod 2 cohomology algebra of the product of two real projective spaces (respectively complex projective spaces), then there does not exist any $\mathbb{Z}_2$-equivariant map from $\mathbb{S}^k \to X$ for $k \geq 2$ (respectively $k \geq 3$), where $\mathbb{S}^k$ is equipped with the antipodal involution.
\end{abstract}

\maketitle

\section{Introduction}
A finitistic space is a paracompact Hausdorff space whose every open covering has a finite dimensional open refinement, where the dimension of a covering is one less than the maximum number of members of the covering which intersect non-trivially. This notion was introduced by Swan \cite{Swan} in his study of fixed point theory. It is a large class of spaces including all compact Hausdorff spaces and all paracompact spaces of finite covering dimension.

It is well known that the class of finitistic spaces is the most suitable for studying topological transformation groups. An excellent account of results in this direction can be seen in \cite{Allday, Bredon2}. Finitistic spaces behave nicely under group actions. More precisely, if $G$ is a compact Lie group acting continuously on a space $X$, then the space $X$ is finitistic if and only if the orbit space $X/G$ is finitistic \cite{Deo2, Deo4}.

Let $G$ be a group acting continuously on a space $X$. Determining the orbit space up to topological or homotopy type is often difficult and hence we try to determine its (co)homological type. For spheres the orbit spaces of free actions of finite groups have been studied extensively by Livesay \cite{Livesay}, Rice \cite{Rice}, Ritter \cite{Ritter}, Rubinstein \cite{Rubinstein} and many others. However, very little is known if the space is a compact manifold other than a sphere. Myers \cite{Myers} determined the orbit spaces of free involutions on three dimensional lens spaces. Tao \cite{Tao} determined the orbit spaces of free involutions on $\mathbb{S}^1 \times \mathbb{S}^2$. Ritter \cite{Ritter2} extended the results of Tao to free actions of cyclic groups of order $2^n$. Recently Dotzel and others \cite{Dotzel2} determined completely the cohomology algebra of orbit spaces of free $\mathbb{Z}_p$ ($p$ prime) and $\mathbb{S}^1$ actions on cohomology product of two spheres. This paper is concerned with the orbit spaces of free involutions on finitistic spaces $X$ having the mod 2 cohomology algebra of the product of two projective spaces.

We write $X \simeq_2 \mathbb{R}P^n \times \mathbb{R}P^m$ if $X$ is a space having the mod 2 cohomology algebra of $\mathbb{R}P^n \times \mathbb{R}P^m$. Similarly, we write $X \simeq_2 \mathbb{C}P^n \times \mathbb{C}P^m$ if $X$ is a space having the mod 2 cohomology algebra of $\mathbb{C}P^n \times \mathbb{C}P^m$. The spaces $\mathbb{R}P^n \times \mathbb{R}P^m$ and $\mathbb{C}P^n \times \mathbb{C}P^m$ are compact Hausdorff spaces and hence are finitistic. Involutions on spaces $X \simeq_2 \mathbb{R}P^n \times \mathbb{R}P^m$ have been studied by Chang and Su \cite{Chang}, where the cohomology structures of the fixed point sets were determined. We study free involutions on such spaces and determine the possible mod 2 cohomology algebra of orbit spaces. If $X/G$ denotes the orbit space, then we obtain the following results.

\begin{theorem}
Let $G=\mathbb{Z}_2$ act freely on a finitistic space $X \simeq_2 \mathbb{R}P^n \times \mathbb{R}P^m$, $1 \leq n \leq m$. Then $H^*(X/G; \mathbb{Z}_2)$ is isomorphic to one of the following graded algebras:
\begin{enumerate}
\item $\mathbb{Z}_2[x,y,z]/ \langle x^2, y^{\frac{n+1}{2}}, z^{m+1} \rangle,$\\
where $deg(x)=1$, $deg(y)=2$, $deg(z)=1$ and $n$ is odd.
\item$\mathbb{Z}_2[x,y,z]/ \langle x^2, y^{n+1}, z^{\frac{m+1}{2}} \rangle,$\\
where $deg(x)=1$, $deg(y)=1$, $deg(z)=2$ and $m$ is odd.
\item$\mathbb{Z}_2[x,y,z,w]/ \langle x^2, y^{\frac{n+1}{2}}, z^{\frac{m+1}{2}}, w^2- \alpha xw - \beta y - \gamma z \rangle,$\\
where $deg(x)=1$, $deg(y)=2$, $deg(z)=2$, $deg(w)=1$ and $\alpha, \beta, \gamma \in \mathbb{Z}_2$ and $n$, $m$ are odd.
\end{enumerate}
\end{theorem}

\begin{theorem}
Let $G=\mathbb{Z}_2$ act freely on a finitistic space $X \simeq_2 \mathbb{C}P^n \times \mathbb{C}P^m$, $1 \leq n \leq m$. Then $H^*(X/G; \mathbb{Z}_2)$ is isomorphic to one of the following graded algebras:
\begin{enumerate}
\item $\mathbb{Z}_2[x,y,z]/ \langle x^3, y^{\frac{n+1}{2}}, z^{m+1} \rangle,$\\
where $deg(x)=1$, $deg(y)=4$, $deg(z)=2$ and $n$ is odd.
\item$\mathbb{Z}_2[x,y,z]/ \langle x^3, y^{n+1}, z^{\frac{m+1}{2}} \rangle,$\\
where $deg(x)=1$, $deg(y)=2$, $deg(z)=4$ and $m$ is odd.
\item$\mathbb{Z}_2[x,y,z,w]/ \langle x^3, y^{\frac{n+1}{2}}, z^{\frac{m+1}{2}}, w^2- \alpha x^2w - \beta y - \gamma z \rangle,$\\
where $deg(x) = 1$, $deg(y)=4$, $deg(z)=4$, $deg(w)=2$ and $\alpha, \beta, \gamma \in \mathbb{Z}_2$ and $n$, $m$ are odd.
\end{enumerate}
\end{theorem}

\section{Free involutions on the product of two projective spaces}
Recall that an involution on a space $X$ is a continuous action of the group $G= \mathbb{Z}_2$ on $X$. We note that the odd dimensional real projective space admits a free involution. Let $m=2n-1$ with $n \geq 1$. Recall that $\mathbb{R}P^m$ is the orbit space of the antipodal involution on $\mathbb{S}^m$ given by $$(x_1,x_2,..., x_{2n-1},x_{2n}) \mapsto (-x_1,-x_2,..., -x_{2n-1},-x_{2n}).$$ If we denote an element of $\mathbb{R}P^m$ by $[x_1,x_2,..., x_{2n-1},x_{2n}]$, then the map $\mathbb{R}P^m \to \mathbb{R}P^m$ given by $$[x_1,x_2,..., x_{2n-1},x_{2n}] \mapsto [-x_2,x_1,..., -x_{2n},x_{2n-1}]$$ defines an involution. It is easy to check that the involution is free.

Similarly, the complex projective space $\mathbb{C}P^m$ admits a free involution when $m \geq 1$ is odd. Recall that $\mathbb{C}P^m$ is the orbit space of the free $\mathbb{S}^1$-action on $\mathbb{S}^{2m+1}$ given by $$(z_1,z_2,..., z_m,z_{m+1}) \mapsto (\zeta z_1, \zeta z_2,..., \zeta z_m,\zeta z_{m+1}) ~ \textrm{for}~ \zeta \in \mathbb{S}^1.$$ Denote an element of $\mathbb{C}P^m$ by $[z_1,z_2,..., z_m,z_{m+1}]$. Then the map $$[z_1,z_2,..., z_m,z_{m+1}] \mapsto [-\overline{z}_2,\overline{z}_1,..., -\overline{z}_{m+1},\overline{z}_m]$$ defines an involution. It is easy to see that the involution is free.

Taking a free involution on $\mathbb{R}P^m$ for $m$ odd (respectively $\mathbb{C}P^m$ for $m$ odd) and any involution on $\mathbb{R}P^n$ (respectively $\mathbb{C}P^n$), the diagonal action gives a free involution on $\mathbb{R}P^n \times \mathbb{R}P^m$ (respectively $\mathbb{C}P^n \times \mathbb{C}P^m$).

\section{Preliminaries for proofs of theorems}
In this section, we recall some facts that we will be using in the paper without mentioning explicitly. For details of the content in this section we refer to \cite{Bredon2, Mccleary}. Throughout we will use \v{C}ech cohomology with $\mathbb{Z}_2$ coefficients and we will suppress it from the cohomology notation. Let the group $G= \mathbb{Z}_2$ act on a space $X$. Let $$G \hookrightarrow E_G \longrightarrow B_G$$ be the universal principal $G$-bundle. Consider the diagonal action of $G$ on $X \times E_G$. Let $$X_G=(X \times E_G) /G$$ be the orbit space of the diagonal action on $X \times E_G$. Then the projection $X \times E_G \to E_G$ is $G$-equivariant and gives a fibration $$X\hookrightarrow X_G \longrightarrow B_G$$ called the Borel fibration \cite[Chapter IV]{Borel3}. Recall that, $$H^*(B_G; \mathbb{Z}_2)= \mathbb{Z}_2[t],$$ where $t$ is a homogeneous element of degree 1. We will exploit the Leray spectral sequence associated to the Borel fibration $X \hookrightarrow X_G \longrightarrow B_G$.

\begin{proposition}\cite[Theorem 5.2]{Mccleary}
Let $X\stackrel{i}{\hookrightarrow} X_G \stackrel{\rho}{\longrightarrow} B_G$ be the Borel fibration. Then there is a first quadrant spectral sequence of algebras $\{E_r^{*,*}, d_r \}$, converging to $H^*(X_G)$ as an algebra, with $$ E_2^{k,l}= H^k(B_G; \mathcal{H}^l(X)),$$ the cohomology of the base $B_G$ with local coefficients in the cohomology of the fiber of $\rho$.
\end{proposition}

\begin{proposition}\cite[Theorem 5.9]{Mccleary}
Let $X\stackrel{i}{\hookrightarrow} X_G \stackrel{\rho}{\longrightarrow} B_G$ be the Borel fibration. Suppose that the system of local coefficients on $B_G$ is simple, then the edge homomorphisms
{\setlength\arraycolsep{35pt}
\begin{eqnarray}
\lefteqn{ H^k(B_G)=E_2^{k,0} \longrightarrow E_3^{k,0}\longrightarrow \cdots  }
                    \nonumber\\
& &   \longrightarrow E_k^{k,0} \longrightarrow E_{k+1}^{k,0}=E_{\infty}^{k,0}\subset H^k(X_G) \nonumber
\end{eqnarray}}
and  $$H^l(X_G) \longrightarrow E_{\infty}^{0,l}= E_{l+1}^{0,l} \subset E_{l}^{0,l} \subset \cdots \subset E_2^{0,l}= H^l(X)$$
are the homomorphisms $$\rho^*: H^k(B_G) \to H^k(X_G) ~ ~ ~ \textrm{and} ~ ~ ~ i^*: H^l(X_G)  \to H^l(X).$$
\end{proposition}

The product in $E_{r+1}^{*,*}$ is induced by the product in $E_{r}^{*,*}$ and the differentials are derivations. The graded commutative algebra $H^*(X_G)$ is isomorphic to Tot$E_{\infty}^{*,*}$, the total complex of $E_{\infty}^{*,*}$. Note also that $H^*(X_G)$ is a $H^*(B_G)$-module with the multiplication given by $$(b,x) \mapsto \rho^*(b) \cup x$$ for $b \in H^*(B_G)$ and $x \in H^*(X_G)$. 

If $G= \mathbb{Z}_2$ act on a space $X \simeq_2 \mathbb{R}P^n \times \mathbb{R}P^m$ or $\mathbb{C}P^n \times \mathbb{C}P^m$ such that the induced action on $H^*(X)$ is trivial, then, by the universal coefficient theorem, $$E_2^{k,l} \cong E_2^{k,0} \otimes E_2^{0,l}.$$  Since $E_2^{0,l}=H^0(B_G; \mathcal{H}^l(X))=H^l(X)^G=H^l(X)$, we have $$E_2^{k,l} \cong H^k(B_G) \otimes H^l(X).$$

We now recall some results regarding $\mathbb{Z}_2$-actions on finitistic spaces.

\begin{proposition}\cite[Chapter VII, Theorem 1.5]{Bredon2}
Let $G = \mathbb{Z}_2$ act freely on a finitistic space $X$. Suppose that $H^j(X) = 0$ for all $j > n$, then $H^j(X_G) =0$ for all $j > n$. 
\end{proposition}

Let $h:X_G \to X/G$ be the map induced by the $G$-equivariant projection $X \times E_G \to X$. Then the following is true.

\begin{proposition}\cite[Chapter VII, Proposition 1.1]{Bredon2}
Let $G=\mathbb{Z}_2$ act freely on a finitistic space $X$. Then $$h^*: H^*(X/G) \stackrel{\cong}{\longrightarrow} H^*(X_G).$$
\end{proposition}

In fact $X/G$ and $X_G$ have the same homotopy type.

\begin{proposition}\cite[Chapter VII, Theorem 1.6.]{Bredon2}
Let $G=\mathbb{Z}_2$ act freely on a finitistic space $X$. Suppose that $\sum_{i\geq0} \textrm{rk}\big( H^i(X) \big)< \infty$ and the induced action on $H^*(X)$ is trivial, then the Leray spectral sequence associated to $X \hookrightarrow X_G \longrightarrow B_G$ does not degenerate.
\end{proposition}

\begin{proposition}\cite[Chapter VII, Theorem 7.4]{Bredon2}
Let $G= \mathbb{Z}_2= \langle T \rangle$ act on a finitistic space $X$ and let $H^i(X)=0$ for all $i > 2n$ and $H^{2n}(X)=\mathbb{Z}_2$. Suppose that $ c\in H^{n}(X)$ is an element such that $cT^*(c) \neq 0$, then the fixed point set is non-empty.
\end{proposition}

\section{Proofs of theorems}
Let $X \simeq_2 \mathbb{R}P^n \times \mathbb{R}P^m$ or $\mathbb{C}P^n \times \mathbb{C}P^m$ be a finitistic space. We first observe for what values of $n$ and $m$, the space $X$ admits a free involution. More precisely, we prove the following for the real case.

\begin{lemma}
Let $G= \mathbb{Z}_2$ act freely on a finitistic space $X \simeq_2 \mathbb{R}P^n \times \mathbb{R}P^m$. Then both $n$ and $m$ cannot be even.
\end{lemma}
\begin{proof}
Suppose both $n$ and $m$ are even. First consider the case when $n < m$. Note that when $l \leq n$, a basis of $H^l(X)$ consists of $$\{ a^l, a^{l-1}b,...,ab^{l-1}, b^l \}.$$ When $n < l \leq m$, a basis of $H^l(X)$ is $$\{ a^nb^{l-n}, a^{n-1}b^{l-n+1}\\,...,ab^{l-1}, b^l \}.$$ And when $m < l \leq n+m $, a basis consists of $$\{ a^nb^{l-n}, a^{n-1}b^{l-n+1},...,\\a^{l-m+1}b^{m-1}, a^{l-m}b^m \}.$$ Therefore the Euler characteristic
{\setlength\arraycolsep{2pt}
\begin{eqnarray*}
\chi(X) & = & \sum_{l= 0}^n (-1)^l \textrm{rk} \big( H^l(X) \big) + \sum_{l=n+1}^m (-1)^l  \textrm{rk}\big( H^l(X) \big) + \sum_{l=m+1}^{n+m} (-1)^l  \textrm{rk} \big( H^l(X) \big)\nonumber\\
& = & \sum_{l= 0}^n (-1)^l (l+1) + \sum_{l=n+1}^m (-1)^l (n+1) + \sum_{l=m+1}^{n+m} (-1)^l(n+m+1-l)
\nonumber\\
& = & \frac{n+2}{2} + 0 - \frac{n}{2} = 1
\nonumber\\
\end{eqnarray*}}
For the case $n=m$, we have either $l \leq n$ or $n < l \leq 2n$. A similar computation gives $\chi(X)=1$. By Floyd's Euler characteristic formula \cite[p.145]{Bredon2}, we have $$\chi(X)+ \chi(X^G)= 2 \chi(X/G).$$ But in case of a free involution $\chi(X^G)=0$ gives a contradiction. Hence at least one of them must be odd.
\end{proof}

Similarly, for the complex case, we prove the following lemma.

\begin{lemma}
Let $G= \mathbb{Z}_2$ act freely on a finitistic space $X \simeq_2 \mathbb{C}P^n \times \mathbb{C}P^m$. Then both $n$ and $m$ cannot be even.
\end{lemma}

Note that the same argument shows that $\mathbb{Z}_2$ cannot act freely on a finitistic space $X \simeq_2 \mathbb{R}P^n$ or $\mathbb{C}P^n$ when $n$ is even.

When a group $G$ acts on a space $X$, there is an action on $H^*(X)$ given by the automorphisms induced by $g^{-1} : X \to  X$ for $g \in G$. Note that it is often difficult to find the induced map on the cohomology even in the case of nice spaces such as spheres. In our context we prove the following.

\begin{proposition}
Let $G= \mathbb{Z}_2$ act freely on a finitistic space $X \simeq_2 \mathbb{R}P^n \times \mathbb{R}P^m$. Then the induced action on $H^*(X)$ is trivial.
\end{proposition}
\begin{proof}
First consider the case when $n = m$. Let $T$ be the generator of $G$ and let $a,~ b \in H^1(X)$ be generators of the cohomology algebra $H^*(X)$. Note that by the naturality of cup product $$T^*(a^ib^j)=(T^*(a))^i(T^*(b))^j$$ for all $i,~ j \geq 0$. Therefore it is enough to consider $$T^*: H^1(X) \to H^1(X).$$ Suppose $T^*$ is not identity, then it cannot preserve both $a$ and $b$. Assuming that $T^*(a) \neq a$, we have $T^*(a) = b$ or $T^*(a) = a+b$. This gives $a^nT^*(a^n)=a^nb^n \neq 0.$ Thus taking $c=a^n$, we have $cT^*(c) \neq 0.$ Therefore by Proposition 3.6 the fixed point set of $T$ is non-empty, which is a contradiction. Hence $T^*$ must be identity.

When $n \neq m$, we have that orders of $a$, $b$ and $a+b$ are $n+1$, $m+1$ and $n+m+1$ respectively. This gives $T^*(a) = a$ and $T^*(b) = b$. Hence in this case also $T^*$ is identity.
\end{proof}

Similarly, for the complex case we have the following.

\begin{proposition}
Let $G= \mathbb{Z}_2$ act freely on a finitistic space $X \simeq_2 \mathbb{C}P^n \times \mathbb{C}P^m$. Then the induced action on $H^*(X)$ is trivial.
\end{proposition}

We now prove the theorems.
\bigskip

\noindent \textbf{Proof of Theorem 1.1.}\\
Let $G= \mathbb{Z}_2$ act freely on a finitistic space $X \simeq_2 \mathbb{R}P^n \times \mathbb{R}P^m$, $1 \leq n \leq m$. By Proposition 4.3, the induced action on the cohomology is trivial. Let $a$, $b \in H^1(X)$ be generators of the cohomology algebra
$$H^*(X) \cong \mathbb{Z}_2[a,b]/ \langle a^{n+1}, b^{m+1}\rangle.$$ The Leray spectral sequence does not degenerate at the $E_2$ term and we have $d_2 \neq 0$ with one of the following:\\
(i) $d_2(1\otimes a)= t^2 \otimes 1$ and $d_2(1\otimes b)= 0$.\\
(ii) $d_2(1\otimes a) =0 $ and $d_2(1\otimes b) = t^2 \otimes 1$.\\
(iii) $d_2(1\otimes a) =t^2 \otimes 1$ and $d_2(1\otimes b)=  t^2 \otimes 1$.\\
We deal with each case separately.\\

\noindent \textbf{The case (i)}\\
Let $d_2(1\otimes a)= t^2 \otimes 1$ and $d_2(1\otimes b)= 0$. Note that by the derivation property of the differential we have
\begin{displaymath}
d_2(1 \otimes a^pb^q) = \left\{ \begin{array}{ll}
t^2 \otimes a^{p-1}b^q & \textrm{if $p$ is odd}\\
0 & \textrm{if $p$ is even.}\\
\end{array} \right.
\end{displaymath}
If $n$ is even, then $a^{n+1}=0$ gives $0= d_2(1 \otimes a^{n+1})= t^2 \otimes a^n$, a contradiction. Hence $n$ must be odd.  Assume that $m$ is also odd. It suffices to compute $$d_2:E_2^{0,l} \to E_2^{2,l-1}.$$ When $l \leq n$, a basis of $E_2^{0,l} \cong H^l(X)$ consists of $$\{ a^l, a^{l-1}b,...,ab^{l-1}, b^l \}.$$ If $l$ is odd, then rk$(ker d_2)= \frac{l+1}{2}=$ rk$(im d_2)$. And if $l$ is even, then rk$(ker d_2)= \frac{l}{2}+1$ and rk$(im d_2)= \frac{l}{2}$.\\
When $n < l \leq m$, a basis consists of $$\{ a^nb^{l-n}, a^{n-1}b^{l-n+1},...,ab^{l-1}, b^l \}.$$ In this case rk$(ker d_2)= \frac{n+1}{2}=$ rk$(im d_2)$.\\
When $m < l \leq n+m $, a basis consists of $$\{ a^nb^{l-n}, a^{n-1}b^{l-n+1},...,a^{l-m+1}b^{m-1}, a^{l-m}b^m \}.$$ If $l$ is odd, then rk$(ker d_2)= \frac{n+m+1-l}{2}=$ rk$(im d_2)$. And if $l$ is even, then rk$(ker d_2)= \frac{n+m-l}{2}$ and  rk$(im d_2)= \frac{n+m+2-l}{2}$.\\
Since $$E_3^{k,l}= ker \{ d_2: E_2^{k,l} \to E_2^{k+2, l-1} \}/ im\{d_2: E_2^{k-2, l+1} \to E_2^{k,l}\},$$ it is clear from the observation above that rk$(E_3^{k,l})$ = 0 and hence $E_3^{k,l}=0$ for all $k \geq 2$ and for all $l$.\\
And  $E_3^{k,l}= ker \{ d_2: E_2^{k,l} \to E_2^{k+2, l-1} \}$ for $k$ = 0, 1 and for all $l$. Note that $$d_r: E_r^{k,l} \to E_r^{k+r, l-r+1}$$
is zero for all $r \geq 3$ as $E_r^{k+r, l-r+1}=0$. Hence $E_{\infty}^{*,*}= E_3^{*,*}$. Since $H^*(X_G) \cong$ Tot$E_{\infty}^{*,*}$, we have $$H^p(X_G)\cong \bigoplus_{i+j=p}E_{\infty}^{i,j}= E_{\infty}^{0,p} \oplus E_{\infty}^{1,p-1}$$ for all $0 \leq p \leq n+m$. Let $x= \rho^*(t) \in E_{\infty}^{1,0}$ be determined by $t \otimes 1 \in E_{2}^{1,0}$.
As $E_{\infty}^{2,0}=0$ we have  $x^2=0$. Note that $1 \otimes a^2 \in E_2^{0,2}$ is a permanent cocycle and therefore it determines an element $u \in E_{\infty}^{0,2}$. Choose $y \in H^2(X_G)$ such that $i^*(y)=a^2$. Then $y$ determines $u$ and satisfies $y^{\frac{n+1}{2}}=0$. Similarly, $1 \otimes b \in E_2^{0,1}$ is a permanent cocycle and  determines an element $v \in E_{\infty}^{0,1}$. Again we choose $z \in H^1(X_G)$ such that $i^*(z)=b$. Then $z$ determines $v$ and  $z^{m+1}=0$. Therefore $$H^*(X_G) \cong \mathbb{Z}_2[x,y,z]/\langle x^2, y^{\frac{n+1}{2}}, z^{m+1}\rangle,$$ where $deg(x)=1$, $deg(y)=2$ and $deg(z)=1$. As the action of $G$ is free, $H^*(X/G ) \cong H^*(X_G)$. The same argument works when $m$ is even. This gives Theorem 1.1 (1). \\

\noindent \textbf{The case (ii)}\\
Let $d_2(1\otimes a)= 0$ and $d_2(1\otimes b)= t^2 \otimes 1$. As above we see that $m$ must be odd.
Rest of the proof is same and we obtain $$H^*(X/G) \cong \mathbb{Z}_2[x,y,z]/\langle x^2, y^{n+1}, z^{\frac{m+1}{2}}\rangle,$$ where $deg(x)=1$, $deg(y)=1$ and $deg(z)=2$. This gives Theorem 1.1 (2).\\

\noindent \textbf{The case (iii)}\\
Let $d_2(1\otimes a) =t^2 \otimes 1$ and $d_2(1\otimes b)=  t^2 \otimes 1$. One can see that 
\begin{displaymath}
d_2(1 \otimes a^pb^q) = \left\{ \begin{array}{ll}
t^2 \otimes a^{p-1}b^q +t^2 \otimes a^pb^{q-1} & \textrm{if $p$ and $q$ are odd}\\
t^2 \otimes a^{p-1}b^q & \textrm{if $p$ is odd and $q$ is even}\\
t^2 \otimes a^pb^{q-1} & \textrm{if $p$ is even and $q$ is odd}\\
0 & \textrm{if $p$ and $q$ are even.}\\
\end{array} \right.
\end{displaymath}
Note that in this case both $n$ and $m$ are odd. As in case (i), by looking at the ranks of $kerd_2$ and $imd_2$ for various values of $l$, we get $E_3^{k,l}=0$ for all $k \geq 2$ and $E_3^{k,l}= ker \{ d_2: E_2^{k,l} \to E_2^{k+2, l-1} \}$ for $k$ = 0, 1. Again $d_r=0$ for all $r \geq 3$ and hence $E_{\infty}^{*,*}= E_3^{*,*}$. We have $$H^p(X_G)\cong \bigoplus_{i+j=p}E_{\infty}^{i,j}= E_{\infty}^{0,p} \oplus E_{\infty}^{1,p-1}$$ for all $0 \leq p \leq n+m$. Let $x= \rho^*(t) \in E_{\infty}^{1,0}$ be determined by $t \otimes 1 \in E_{2}^{1,0}$. As $E_{\infty}^{2,0}=0$ we have  $x^2=0$. Note that $1 \otimes a^2 \in E_2^{0,2}$ is a permanent cocycle and therefore yields an element $u \in E_{\infty}^{0,2}$. Choose $y \in H^2(X_G)$ such that $i^*(y)=a^2$. Then $y$ determines $u$ and satisfies $y^{\frac{n+1}{2}}=0$. Similarly, $1 \otimes b^2 \in E_2^{0,2}$ is a permanent cocycle and determines an element $v \in E_{\infty}^{0,2}$. Again we choose $z \in H^2(X_G)$ such that $i^*(z)=b^2$. Then $z$ determines $v$ and $z^{\frac{m+1}{2}}=0$. Also $1 \otimes (a+b) \in E_2^{0,1}$ is a permanent cocycle and yields an element $s \in E_{\infty}^{0,1}$. Let $w \in H^1(X_G)$ such that $i^*(w)= a+b$. Then $w$ determines $s$. Note that $$w^2=\alpha xw + \beta y + \gamma z$$ for some $\alpha, \beta, \gamma \in \mathbb{Z}_2$. Hence $$H^*(X_G) \cong \mathbb{Z}_2[x,y,z,w]/\langle x^2, y^{\frac{n+1}{2}}, z^{\frac{m+1}{2}}, w^2 -\alpha xw - \beta y - \gamma z\rangle,$$ where $deg(x)=1$, $deg(y)=2$, $deg(z)=2$ and $deg(w)=1$. As the action of $G$ is free, $H^*(X/G ) \cong H^*(X_G)$ which gives Theorem 1.1(3). \hfill $\Box$

\begin{corollary}
Let $G=\mathbb{Z}_2$ act freely on a finitistic space $X \simeq_2 \mathbb{R}P^n$, where $n \geq 1$ is odd. Then $$H^*(X/G) \cong \mathbb{Z}_2[x,y]/ \langle x^2, y^{\frac{n+1}{2}}\rangle,$$
where $deg(x)=1$ and $deg(y)=2$.
\end{corollary}

\begin{proof}
The proof follows by taking $b=0$ in the proof of Theorem 1.1. We have $d_2(1 \otimes a) =t^2 \otimes 1$ and 
\begin{displaymath}
d_2(1 \otimes a^p) = \left\{ \begin{array}{ll}
t^2 \otimes a^{p-1} & \textrm{if $p$ is odd}\\
0 & \textrm{if $p$ is even.}\\
\end{array} \right.
\end{displaymath}
This gives $$d_2:E_2^{k,l} \to E_2^{k+2,l-1}$$ is an isomorphism for $l$ odd and zero for $l$ even. Hence $E_3^{k,l}=0$ for all $k \geq 2$. Also $$d_r: E_r^{k,l} \to E_r^{k+r, l-r+1}$$ is zero for all $r \geq 3$ as $E_r^{k+r, l-r+1}=0$ for all $r \geq 3$. Imitating the above proof we get $$H^*(X/G; \mathbb{Z}_2) \cong \mathbb{Z}_2[x,y]/ \langle x^2, y^{\frac{n+1}{2}}\rangle,$$
where $deg(x)=1$ and $deg(y)=2$.
\end{proof}

\begin{remark}
Taking $n=m=1$, we have $X\simeq_2 \mathbb{R}P^n \times \mathbb{R}P^m = \mathbb{S}^1 \times \mathbb{S}^1$ and hence $$H^*(X/G) \cong \mathbb{Z}_2[x,w]/\langle x^2, w^2 - \alpha xw\rangle,$$ where $deg(x)=1$, $deg(w)=1$ and $\alpha \in \mathbb{Z}_2$, which is consistent with \cite[Theorem 2 (iii)]{Dotzel2}.
\end{remark}
\bigskip

\noindent \textbf{Proof of Theorem 1.2.}\\
Since the proof is analogous to that of Theorem 1.1, we describe it rather briefly. Let $G= \mathbb{Z}_2$ act freely on a finitistic space $X \simeq_2 \mathbb{C}P^n \times \mathbb{C}P^m$, $1 \leq n \leq m$. By Proposition 4.4, the induced action on the cohomology is trivial. Note that $E_2^{k,l}= 0$ for $l$ odd. This gives $$d_2:E_2^{k,l} \to E_2^{k+2,l-1}$$ is zero and hence $E_3^{k,l}= E_2^{k,l}$ for all $k$, $l$. Let $a$, $b \in H^2(X)$ be generators of the cohomology algebra $H^*(X)$. Since the Leray spectral sequence does not degenerate, we have $d_3 \neq 0$ with one of the following:\\
(i) $d_3(1\otimes a)= t^3 \otimes 1$ and $d_3(1\otimes b)= 0$.\\
(ii) $d_3(1\otimes a) =0 $ and $d_3(1\otimes b) = t^3 \otimes 1$.\\
(iii) $d_3(1\otimes a) =t^3\otimes 1$ and $d_3(1\otimes b)=  t^3 \otimes 1$.\\
Again we proceed case by case.\\

\noindent \textbf{The case (i)}\\
Let $d_3(1\otimes a)= t^3 \otimes 1$ and $d_3(1\otimes b)= 0$, then
\begin{displaymath}
d_3(1 \otimes a^pb^q) = \left\{ \begin{array}{ll}
t^3 \otimes a^{p-1}b^q & \textrm{if $p$ is odd}\\
0 & \textrm{if $p$ is even.}\\
\end{array} \right.
\end{displaymath}
Note that $n$ must be odd. For various values of $l$ we consider the differentials $$d_3:E_3^{0,2l} \to E_3^{3,2l-2}$$ as in the proof of Theorem 1.1 and compute the ranks of $kerd_3$ and $im d_3$. Since $$E_4^{k,2l}= ker \{ d_3: E_3^{k,2l} \to E_3^{k+3, 2l-2} \}/ im\{d_3: E_3^{k-3, 2l+2} \to E_3^{k,2l}\},$$ it is clear that rk$(E_4^{k,2l})$ = 0 and hence $E_4^{k,2l}=0$ for all $k \geq 3$.\\
And  $E_4^{k,2l}= ker \{ d_3: E_3^{k,2l} \to E_3^{k+3, 2l-2} \}$ for $k$ = 0, 1, 2. Note that $E_4^{k,l}=0$ for all $k$ and for all odd $l$. Observe that $$d_r: E_r^{k,l} \to E_r^{k+r, l-r+1}$$
is zero for all $r \geq 4$ as $E_r^{k+r, l-r+1}=0$ for all $r \geq 4$. Hence $E_{\infty}^{*,*}= E_4^{*,*}$. Since $H^*(X_G) \cong$ Tot$E_{\infty}^{*,*}$, have $$H^p(X_G)\cong \bigoplus_{i+j=p}E_{\infty}^{i,j}= E_{\infty}^{0,p} \oplus E_{\infty}^{1,p-1} \oplus E_{\infty}^{2,p-2}$$ for all $0 \leq p \leq 2(n+m)$. Let $x= \rho^*(t) \in E_{\infty}^{1,0}$ be determined by $t \otimes 1 \in E_{2}^{1,0}$.
As $E_{\infty}^{3,0}=0$ we have  $x^3=0$. Note that $1 \otimes a^2 \in E_2^{0,4}$ is a permanent cocycle and determines an element $u \in E_{\infty}^{0,4}$. Choose $y \in H^4(X_G)$ such that $i^*(y)=a^2$. Then $y$ determines $u$ and satisfies $y^{\frac{n+1}{2}}=0$. Similarly, $1 \otimes b \in E_2^{0,2}$ is a permanent cocycle and gives an element $v \in E_{\infty}^{0,2}$. Again we choose $z \in H^2(X_G)$ such that $i^*(z)=b$. Then $z$ determines $v$ and $z^{m+1}=0$. Therefore $$H^*(X_G) \cong \mathbb{Z}_2[x,y,z]/\langle x^3, y^{\frac{n+1}{2}}, z^{m+1}\rangle,$$ where $deg(x)=1$, $deg(y)=4$ and $deg(z)=2$. As the action of $G$ is free, $H^*(X/G )\cong H^*(X_G)$, which gives Theorem 1.2 (1).\\

\noindent \textbf{The case (ii)}\\
The proof is similar to that of case (i) and we get Theorem 1.2 (2).\\

\noindent \textbf{The case (iii)}\\
Let $d_3(1\otimes a) =t^3 \otimes 1$ and $d_3(1\otimes b)=  t^3 \otimes 1$. One can see that 
\begin{displaymath}
d_3(1 \otimes a^pb^q) = \left\{ \begin{array}{ll}
t^3 \otimes a^{p-1}b^q +t^3 \otimes a^pb^{q-1} & \textrm{if $p$ and $q$ are odd}\\
t^3 \otimes a^{p-1}b^q & \textrm{if $p$ is odd and $q$ is even}\\
t^3 \otimes a^pb^{q-1} & \textrm{if $p$ is even and $q$ is odd}\\
0 & \textrm{if $p$ and $q$ are even.}\\
\end{array} \right.
\end{displaymath}
Again in this case both $n$ and $m$ are odd. As in case (i) we see that $E_4^{k,2l}=0$ for all $k \geq 3$ and $E_4^{k,2l}= ker \{ d_3: E_3^{k,2l} \to E_3^{k+3, 2l-2} \}$ for $k$ = 0, 1, 2.  Note that $E_4^{k,l}=0$ for all $k$ and for all odd $l$. Thus $d_r=0$ for all $r \geq 4$ and hence $E_{\infty}^{*,*}= E_4^{*,*}$. Again we have $$H^p(X_G)\cong \bigoplus_{i+j=p}E_{\infty}^{i,j}= E_{\infty}^{0,p} \oplus E_{\infty}^{1,p-1} \oplus E_{\infty}^{2,p-2}$$ for all $0 \leq p \leq 2(n+m)$. Let $x= \rho^*(t) \in E_{\infty}^{1,0}$ be determined by $t \otimes 1 \in E_{2}^{1,0}$. As $E_{\infty}^{3,0}=0$ we have  $x^3=0$. Note that $1 \otimes a^2 \in E_2^{0,4}$ is a permanent cocycle and gives an element $u \in E_{\infty}^{0,4}$. Let $y \in H^4(X_G)$ such that $i^*(y)=a^2$. Then $y$ determines $u$ and  $y^{\frac{n+1}{2}}=0$. Similarly, $1 \otimes b^2 \in E_2^{0,4}$ is a permanent cocycle and yields an element $v \in E_{\infty}^{0,4}$. Choose $z \in H^4(X_G)$ such that $i^*(z)=b^2$. Then $z$ determines $v$ and satisfies $z^{\frac{m+1}{2}}=0$. Also $1 \otimes (a+b) \in E_2^{0,2}$ is a permanent cocycle and gives an element $s \in E_{\infty}^{0,2}$. Choose $w \in H^2(X_G)$ such that $i^*(w)= a+b$. Then $w$ determines $s$. We note that $$w^2=\alpha x^2w + \beta y + \gamma z$$ for some $\alpha, \beta, \gamma \in \mathbb{Z}_2$. Hence $$H^*(X_G) \cong \mathbb{Z}_2[x,y,z,w]/\langle x^3, y^{\frac{n+1}{2}}, z^{\frac{m+1}{2}}, w^2 -\alpha x^2w - \beta y - \gamma z\rangle,$$ where $deg(x)=1$, $deg(y)=4$, $deg(z)=4$ and $deg(w)=2$. Since $H^*(X/G ) \cong H^*(X_G)$, we get Theorem 1.2 (3). \hfill $\Box$

\begin{corollary}
Let $G=\mathbb{Z}_2$ act freely on a finitistic space $X \simeq_2 \mathbb{C}P^n$, where $n \geq 1$ is odd. Then $$H^*(X/G) \cong \mathbb{Z}_2[x,y]/ \langle x^3, y^{\frac{n+1}{2}}\rangle,$$
where $deg(x)=1$ and $deg(y)=4$
\end{corollary}

\begin{proof}
The proof follows by taking $b=0$ in the proof of Theorem 1.2. We have $d_2=0$ and $d_3(1 \otimes a) =t^3 \otimes 1$ and 
\begin{displaymath}
d_3(1 \otimes a^p) = \left\{ \begin{array}{ll}
t^3 \otimes a^{p-1} & \textrm{if $p$ is odd}\\
0 & \textrm{if $p$ is even.}\\
\end{array} \right.
\end{displaymath}
This gives $$d_3:E_3^{k,2l} \to E_3^{k+3,2l-2}$$ is an isomorphism for $l$ odd and zero for $l$ even. Hence $E_4^{k,l}=0$ for all $k \geq 3$ and $$d_r: E_r^{k,l} \to E_r^{k+r, l-r+1}$$ is zero for all $r \geq 4$ as $E_r^{k+r, l-r+1}=0$ for all $r \geq 4$. Imitating the above proof we get $$H^*(X/G; \mathbb{Z}_2) \cong \mathbb{Z}_2[x,y]/ \langle x^3, y^{\frac{n+1}{2}}\rangle,$$
where $deg(x)=1$ and $deg(y)=4$.
\end{proof}

\begin{remark}
Taking $n=m=1$, we have $X \simeq_2 \mathbb{C}P^n \times \mathbb{C}P^m = \mathbb{S}^2 \times \mathbb{S}^2$ and hence $$H^*(X/G) \cong \mathbb{Z}_2[x,w]/\langle x^3, w^2 - \alpha x^2w\rangle,$$ where $deg(x)=1$, $deg(w)=2$ and $\alpha \in \mathbb{Z}_2$, which is consistent with \cite[Theorem 2 (iii)]{Dotzel2}.
\end{remark}

\begin{example}
For an odd integer $n$, take a free involution on $\mathbb{R}P^n$ and for any positive integer $m$ take the trivial action on $\mathbb{R}P^m$. Then the diagonal action on $X= \mathbb{R}P^n \times \mathbb{R}P^m$ is a free involution.
By definition $\mathbb{R}P^n= \mathbb{S}^n/{\mathbb{Z}_2}$, where $\mathbb{S}^n$ is equipped with the antipodal action of $\mathbb{Z}_2$. Now given the free involution on $\mathbb{R}P^n$, lifting the action gives a free and orthogonal action of a group $H$ of order 4 on $\mathbb{S}^n$. It is well known that $\mathbb{Z}_2 \oplus\mathbb{Z}_2$ cannot act freely on $\mathbb{S}^n$ \cite[Chapter III, Theorem 8.1]{Bredon2}. Hence $H$ must be the cyclic group of order 4 acting freely and orthogonally on $\mathbb{S}^n$ and therefore $\mathbb{R}P^n/ \mathbb{Z}_2 = \mathbb{S}^n/H = L^n(4,1)$, the $n$-dimensional Lens space whose cohomology algebra is known to be $$\mathbb{Z}_2[x,y]/\langle x^2, y^{\frac{n+1}{2}} \rangle,$$
where  $deg(x)=1$ and $deg(y)=2$. Hence the cohomology algebra of the orbit space $X/ \mathbb{Z}_2 = L^n(4,1) \times \mathbb{R}P^m$ is $$\mathbb{Z}_2[x,y,z]/\langle x^2, y^{\frac{n+1}{2}}, z^{m+1} \rangle,$$ where  $deg(x)=1$, $deg(y)=2$ and $deg(z)=1$. This realizes Theorem 1.1(1). Interchanging $n$ and $m$ realizes Theorem 1.1(2).
\end{example}

\begin{example}
For the complex case, if we take $n=1$, then $\mathbb{C}P^1=\mathbb{S}^2$. The orbit space of any free involution on $\mathbb{S}^2$ is $\mathbb{R}P^2$, whose cohomology algebra is $\mathbb{Z}_2[x]/\langle x^3 \rangle$. For any positive integer $m$, the trivial action on $\mathbb{C}P^m$ gives a free involution on $X=\mathbb{C}P^1 \times \mathbb{C}P^m$
whose orbit space is $X/ \mathbb{Z}_2 =\mathbb{R}P^2 \times \mathbb{C}P^m$ and has cohomology algebra $$\mathbb{Z}_2[x,z]/\langle x^3, z^{m+1} \rangle,$$ where  $deg(x)=1$ and $deg(z)=2$. This realizes Theorem 1.2 (1). Similarly, one can realize Theorem 1.2(2).
\end{example}

\section{Application to equivariant maps}
Let $\mathbb{S}^k$ be the unit $k$-sphere equipped with the antipodal involution and $X$ be a paracompact Hausdorff space with a fixed free involution. We give an application of our results to non-existence of $\mathbb{Z}_2$-equivariant maps from $\mathbb{S}^k \to X$. Conner and Floyd defined the index of the involution on $X$ as $$\textrm{ind}(X) = max ~ \{~ k ~|~ \textrm{there exist a}~ \mathbb{Z}_2 \textrm{-equivariant map}~ \mathbb{S}^k \to X \}.$$ It is natural to consider a purely cohomological criteria to study the above invariant. The best known and most easily managed cohomology classes are the characteristic classes with coefficients in $\mathbb{Z}_2$. Generalizing the Yang's index \cite{Yang2}, Conner and Floyd defined
$$\textrm{{co-ind}}_{\mathbb{Z}_2}(X)~ = ~\textrm{largest integer}~k~ \textrm{such that}~ w^k \neq 0,$$
where $w \in H^1(X/G; \mathbb{Z}_2 )$ is the Whitney class of the principal $G$-bundle $$X \to X/G.$$ Since $\textrm{{co-ind}}_{ \mathbb{Z}_2}(\mathbb{S}^k)$ = $k$, by \cite[(4.5)]{Conner}, we have
$$\textrm{ind}(X)\leq \textrm{{co-ind}}_{\mathbb{Z}_2}(X).$$
Also, since $X$ is paracompact Hausdorff, we can take a classifying map $$c : X/G \to B_G$$ for the principal $G$-bundle $X \to X/G$. If $h: X/G \to X_G$ is a homotopy equivalence, then $\rho h : X/G \to B_G$ also classifies the principal $G$-bundle $X \to X/G$ and hence it is homotopic to $c$. Therefore it suffices to consider the map $$\rho^*: H^1(B_G) \to H^1(X_G).$$ The image of the Whitney class of the universal principal $G$-bundle $G \hookrightarrow E_G \longrightarrow B_G$ is the Whitney class of $X \to X/G$.

Now, when $X \simeq_2 \mathbb{R}P^n \times \mathbb{R}P^m$ is a finitistic space, from the proof of Theorem 1.1, we have that $x \in H^1(X/G)$ is the Whitney class with $x \neq 0$ and $x^2=0$. This gives $\textrm{{co-ind}}_{\mathbb{Z}_2}(X) = 1$ and $\textrm{ind}(X)\leq 1$. Hence there is no $\mathbb{Z}_2$-equivariant map from $\mathbb{S}^k \to X$ for $k \geq 2$.

Similarly, when $X \simeq_2 \mathbb{C}P^n \times \mathbb{C}P^m$ is a finitistic space, from the proof of Theorem 1.2, $x \in H^1(X/G)$ is the Whitney class with $x^2 \neq 0$ and $x^3=0$. This gives $\textrm{{co-ind}}_{\mathbb{Z}_2}(X) = 2$ and $\textrm{ind}(X)\leq 2$. Hence there is no $\mathbb{Z}_2$-equivariant map from $\mathbb{S}^k \to X$ for $k \geq 3$.

\subsection*{Acknowledgment}
The author thanks the referee for many valuable suggestions which improved the presentation and exposition of the paper.

\bibliographystyle{amsplain}

\end{document}